\documentclass[11pt]{article}
\usepackage{amssymb}
\usepackage{amsmath,amsthm}
\usepackage[latin1]{inputenc}
\usepackage{graphicx}
\usepackage{multicol}
\usepackage{wrapfig}
\usepackage{hyperref}
\usepackage{geometry}
\usepackage{color}

\newtheorem{theorem}{Theorem}
\newtheorem{lemma}[theorem]{Lemma}
\newtheorem{prop}[theorem]{Proposition}
\newtheorem{obs}{Observation}

\newtheorem{cor}{Corollary}
\newtheorem{dfn}{Definition}

\newcommand{\Mod}[1]{\ (\text{mod}\ #1)}
\begin{document}

\title{Permutation graphs and unique games}
\author{ Monika Rosicka\thanks{%
Faculty of Mathematics, Physics and Informatics, University of Gda\'{n}sk,
80-952 Gda\'{n}sk, Poland; Institute of Theoretical Physics and
Astrophysics, and National Quantum Information Centre in Gda\'{n}sk, 81-824
Sopot, Poland. \texttt{rosamo@op.pl}} \and Simone Severini\thanks{%
Department of Computer Science, University College London, WC1E 6BT London,
United Kingdom. \texttt{s.severini@ucl.ac.uk}}}
\maketitle

\begin{abstract}
We study the value of unique games as a graph-theoretic parameter. This is
obtained by labeling edges with permutations. We describe the classical 
value of a game as well as give a necessary and
sufficient condition for the existence of an optimal assignment based on a
generalisation of permutation graphs and graph bundles. In considering some
special cases, we relate XOR games to \textsc{Edge Bipartization}, and
define an edge-labeling with permutations from Latin squares. 
\end{abstract}
 
\section{Introduction}

We shall work with undirected graphs $G=(V,E)$ without self-loops or
multiple edges. An edge between vertices $u$ and $v$ is denoted by $\{u,v\}$
or $uv$. A \emph{vertex-labeling}, $l:V(G)\mapsto S$, is an assignment of
elements of a given set $S$ to the vertices of a graph $G$. An \emph{%
edge-labeling}, $L:E(G)\mapsto S^{\prime }$, is an assignment of elements of
a given set $S^{\prime }$ to the edges of $G$. A \emph{permutation} of a set 
$S=\{x_{1},...,x_{n}\}$ is a bijective function $\pi :S\mapsto S$. As usual,
a permutation $\pi $ can be defined by a set of ordered pairs $\{(x_{1},\pi
(x_{1})),...,(x_{n},\pi (x_{n}))\}$. A \textit{fixed point} of a permutation 
$\pi $ is an element $x\in S$ such that $\pi (x)=x$. By $%
(x_{1}x_{2}...x_{n}) $, we denote a permutation $\pi $ such that $\pi
(x_{1})=x_{2},...,\pi (x_{n-1})=x_{n},\pi (x_{n})=x_{1}$. The set of all
permutations on $[n]=\{0,...,n-1\}$ is denoted by $S_{n}$. We consider
edge-labelings with permutations. For undirected graphs, we generally
consider labelings with transpositions only, \emph{i.e.}, permutations $\pi $
such that $\pi ^{-1}=\pi $. For directed graphs we may use any subset of $%
S_{n}$.

Let $G$ be a graph and $K:E(G)\mapsto S_{n}$ an edge-labeling assigning a
permutation $\pi _{i}:[n]\mapsto \lbrack n]$ to every edge $e_{i}\in E(G)$.
A \emph{contradiction} in a vertex-labeling $k:V(G)\mapsto \lbrack n]$ is an
edge $uv$ such that $\pi (k(u))\neq k(v)$, where $\pi $ is the permutation
assigned to $uv$. For a given graph $G$ and an edge-labeling $K:E(G)\mapsto
S_{n}$, we say that a vertex-labeling $k:V\mapsto \lbrack n]$ is \emph{%
consistent} if it gives no contradictions. In such a case, the edge-labeling 
$K$ is said to be \emph{vertex-proper}.


The \emph{contradiction number}, $\beta _{C}(G,K)$, is the minimum number of
contradictions for a given graph $G$ and an edge-labeling $K:E\mapsto S_{n}$%
, where the minimum is taken over all possible vertex-labelings: 
\begin{equation*}
\beta _{C}(G,K)=\min_{k}\left\vert \{e_{i}=uv\in E(G):\pi _{i}(k(u))\neq
k(v),\pi _{i}=K(e_{i})\}\right\vert .
\end{equation*}%
For a given labeled graph $(G,K:E\mapsto S_{n})$, we say that a vertex-labeling $K:V\mapsto
\lbrack n]$ is \emph{optimal} if it gives exactly $\beta _{C}(G,K)$
contradictions. It is clear that every consistent vertex-labeling is
optimal. For a given graph $G$ and an edge-labeling $K:E\mapsto S_{n}$, the 
\emph{assignment number}, $\beta _{C}^{\prime }(G,K)$, is the number of
consistent vertex-labelings possible for $G$ and $K$.

In \cite{qqq}, labeled graphs are used in the study of \emph{contextuality}. When investigating contextuality, we attempt to quantify how much the outcome of a measurement of a physical observable depends on the context in which it is measured. By a \emph{context}, we understand a set of observables which can be measured at the same time. By a \emph{contextual game}, we mean a scenario in which two players are asked to assign values $a$ and $b$ to certain randomly chosen variables $x\in X$ and $y\in Y$, respectively, under some requirements. Neither of the players knows which variable was chosen for the other one and they are not allowed to communicate with each other during the game. For each pair, $x, y$, of variables which may be chosen, we want the values $a$ and $b$ to satisfy certain constraints. If $a$ and $b$ satisfy the constraints, both players win, otherwise they both lose. The maximum probability of winning a given game with a classical strategy is called the \emph{classical value} of the game and denoted by $\omega$. Similarly, the \emph{quantum value} of the game, denoted by $\omega^*$, is the highest probability of winning the game using quantum resources, commonly entanglement shared between the players. 

A \emph{unique game} (see, for example, \cite{Tre}) is a contextual game in which the variables can take values from the set $[n]$ and the constraints are in the form $\pi_{xy}(a)=b,$ where $a$ and $b$ are values assigned to the variables $x$ and $y$, respectively, and $\pi$ is a permutation of $[n].$ 

Classically, the optimal strategy for any unique game is to deterministically assign values to all of the variables so as to satisfy as many of the constraints as possible. Every constraint is satisfied with probability either $0$ or $1$.
Quantum theory allows for strategies in which each constraint is satisfied with some probability $0\leq p_{xy}\leq 1$. In some cases, the quantum value of a game may be higher than its classical value, because the players share entanglement. 

A unique game with $n$ possible answers can be defined in terms of a graph $G$ with an edge-labeling $K\colon E(G)\mapsto S_n$. The set of variables corresponds to $V(G)$. Two vertices $x$ and $y$ of $G$ are adjacent if and only if we may ask about the values of both at the same time. The edge-labeling $K$ represents the constraints. If the answers for a given pair $x,y$ should satisfy $\pi_{xy}(a)=b,$ then $K(xy)=\pi_{xy}.$ The most common scenario is one where the sets $X$ and $Y$ of questions each player may be asked are disjoint. These games are represented by bipartite labeled graphs. However, single-player games on non-bipartite graphs are also considered. Thus, the classical value of a game can be written as follows:
\begin{center}
$\omega(G,K) = \max_{P \in C} \frac{1}{|E(G)|} \sum_{\{x,y\} \in E } V(a,b|x,y) P(a,b|x,y)$
\end{center}
where $V(a,b|x,y)$ is the indicator function, \emph{i.e.} $V(a,b|x,y)=1$ if $\pi_{xy}(a)=b$ and $V(a,b|x,y)=0$, otherwise. Hence, the quantum value of a game, denoted by $\omega^*(G,K)$, is defined in exactly the same way, but with the maximum taken over all strategies involving shared entanglement. For the purpose of this paper, we do need to go beyond this deliberately vague definition. 

It is easy to see that a classical strategy, or deterministic assignment of values to all variables in the game, is the same as a vertex-assignment of the labeled graph representing the game. Thus, the classical value of the game can be given in terms of the contradiction number of the labeled graph representing the game: 

\begin{center}
$\omega(G,K)=1-\frac{\beta_C(G,K)}{\left|E(G)\right|}.$ 
\end{center}

The simplest and best-known class of unique games are \textit{XOR games}, in which the variables can only take the values $0$ and $1.$ The constraints in those games are given in the form $a+b \equiv i \Mod{2}$, where $i\in\{0,1\}$. XOR games can be described in terms of signed graphs and thus the classical value of an XOR game is connected to the notion of balance in a signed graph.  

Labeled graph equivalence, defined in section \ref{equiv}, allows us to introduce an equivalence for unique games. Two unique games are considered equivalent if they are associated to equivalent labeled graphs. Since equivalence preserves the contradiction number of a labeled graph, it is clear that equivalent games have the same classical value. The quantum value is also preserved, as graph isomorphisms and switches represent the relabeling of the inputs $x,y$ and outputs $a,b$, respectively. Thus, equivalent games have the same classical and quantum values.

The permutation graph $KG=\overline{K_{n}}\bowtie ^{K }G$, which we describe in more detail in Section 2, was first introduced with the purpose of studying the contradiction number of labeled graphs $(G,K)$ and classical value of the corresponding games. However,  the structure and components of $KG$ also provide information about optimal quantum strategies. Let $(G,K\colon E(G)\mapsto S_n)$ be a connected labeled graph such that $KG$ is not a connected graph. If $K'\colon E(G)\mapsto S_d$, for $d<n$, is an edge-labeling of $G$, such that $K'$ is a component of $G$, then any strategy for the game defined by $(G,K')$ is also a valid strategy for $(G,K).$ It follows that $\omega^*(G,K)\geq\omega^*(G,K')$. However, in many cases, $\omega^*(G,K)$ is equal to the maximum value $\omega^*(G,K')$ over all labelings $K'$ such that $K'G$ is a component of $KG.$ The structure and components of the permutation graph $KG$ will most likely play an important role in future studies of unique games.

In this paper, we focus on the contradiction and assignment numbers of labeled graphs in general as well as graphs whose labels are given by specific classes of permutations. In Section 2, we introduce the permutation graph $KG$ obtained from a labeled graph $(G,K)$ and show how it can be used to study consistent labelings. In Section 3, we study the effects of some operations, such as deleting an edge or identifying two vertices, on the contradiction and assignment numbers of a labeled graph. In Section 4, we define an equivalence relation for labeled graphs as a generalized version of signed graph equivalence and show that if two labeled graphs are equivalent, then their contradiction and assignment numbers are equal. In sections 5 and 6 we focus on labelings with specific classes of permutations. While this has not immediate practical interest, it is a playground to introduce new mathematical notions and gain a better understanding of special cases. In particular, we consider, the set $L_n=\{\pi_i\in S_n: \pi_i(x) \equiv i-x \Mod{n}$ for $x, i\in [n]\}$, $x+y \equiv i \Mod{2}$, which corresponds to the class of \textit{GXOR games}, studied in \cite{qqq}, and $S_2$, which we compare to signed graphs. 

We begin with some observations on labeled trees and cycles that arguably give simplest cases.

\begin{obs}
For any tree all edge-labelings are vertex-proper.
\end{obs}

\begin{proof}
Let $T$ be a tree, and let $K:E(T)\mapsto S_n$ be an edge-labeling. Given the tree property, any two vertices $u,v\in V(T)$ are connected by exactly one path. For any path $P_t=(v_1,...,v_t)$, we can define a permutation $\pi_{P_t}=K(v_{t-1}v_t)...K(v_1v_2).$ Let $v$ be any vertex of $T.$ Assign the value $k(v)$ to the vertex $v.$ Then, for every $u\in V(T)-\{v\},$ let $k(u)=\pi_{P_t}(k(v)),$ where $P_t$ is the path connecting $u$ to $v.$ This assignment is consistent.
\end{proof}

It is clear that a graph with only one cycle (often called unicyclic) can have at most one
contradiction. Whether or not
there is a contradicion in the cycle can be determined through the
composition of all permutations assigned to the edges of the cycle. For a
given cycle $C_{t}=(v_{1},...,v_{t})$, we define a permutation $\pi
_{C_{t}}=K(e_{1})K(e_{2})...K(e_{t})$, where $e_{i}\cap e_{i+1}=\{v_{i}\}$,
for all $i$ and $v_{t}=v_{0}.$

\begin{theorem}
A cycle $C_{t}$ has a consistent vertex-labeling for a given edge-labeling $%
K $ if and only if $\pi _{C_{t}}$ has at least one fixed point.
\end{theorem}

\begin{proof}
Suppose $k:V(C)\mapsto [n]$ is a consistent vertex-assignment of $C.$ For any $v_i\in C,$ $k(v_{i+1})=K(e_{i+1})(k(v_i))=K(e_1)K(e_2)...K(e_{i+1})(k(v_0)).$ Since $v_n=v_0,$ and therefore $k(v_t)=k(v_0),$ we have $\pi_c(k(v_0))=k(v_0).$  
\end{proof}

\begin{cor}
The number of fixed points of $\pi _{C_{t}}$ is equal to the number of
consistent vertex-labelings of $C_{t}.$
\end{cor}

It follows that the number of contradictions in a given graph is at most the
number of cycles. It may, however, be greater than the number of cycles
containing contradictions.

\bigskip




\section{Permutation graphs}

To study the contradiction and assignment numbers in a given graph $G$ with
edge-labeling $K:E(G)\mapsto \Pi \subset S_{n}$, we define the graph $KG=%
\overline{K_{n}}\bowtie ^{K}G$. The \emph{permutation graph} $G\bowtie
^{\Pi }H$ of two graphs $G$ and $H$ is as follows. 

\begin{dfn}
Let $G$ and $H$ be two graphs and let $\Pi\colon E(H)\mapsto S_n$ 
(where $n=\left|V(G)\right|$) be a function assigning a permutation $\pi_{ij}$ 
of $V(G)$ to each edge $v_iv_j$ of $H$.

\begin{enumerate}
\item Let $G_1,...,G_n$ be copies of $G$. For any vertex $v\in V(G),$ we denote 
the copy of $u\in V(G)$ in $V_i=V(G_i)$ as $u_i$ and for any set of vertices 
$S\subseteq V(G),$ the copy of $S$ in $G_i$ is denoted by $S_i.$

\item Two vertices $u_i\in V_i$ and $w_j\in V_j$ (i.e. copies of $u$ in $V_i$ and 
$w$ in $V_j$, $i<j$) are adjacent in $G\bowtie^{\Pi}H$ if and only if $ij\in E(H)$
 and $\pi_{ij}(u)=w$, where $\pi_{ij}=\Pi(v_iv_j).$
\end{enumerate}
\end{dfn}

In \cite{LS}, this type of graph is referred to as a \emph{permutation graph over
a graph} $H$. Every edge-labeling $K:E\mapsto S_{n}$ of a graph $G$ gives rise to a
graph $KG=\overline{K_{n}}\bowtie^{K}G,$ which can also be defined as follows. 

\begin{dfn}
Let $(G,K\colon E(G)\mapsto S_n)$ be a labeled graph.

\begin{enumerate}
\item For every vertex $v_i\in V(G)$ take an independent set $V_i=\{v_{i0},...v_{in-1}\}.$

\item Two vertices $v_{ij}\in V_j, v_{st}\in V_s$ are adjacent in $KG$ if and only if 
$v_iv_s\in E(G)$ and $\pi_{is}(j)=t,$ where $\pi_{is}=K(v_iv_s).$
\end{enumerate}
\end{dfn}

This object is strictly linked to the notion of a
permutation voltage graph originally introduced in \cite{GT}, and it can be
seen as a generalization of permutation graphs and graph bundles \cite{MPS}. 

\begin{lemma}
Let $K^{\prime }$ be an edge-labeling of the graph $KG$ such that all edges
in $E(KG)$ corresponding to a certain $e\in E(G)$ are labeled with the same
permutation as $e$. The edge-labeling $K^{\prime }$ is vertex-proper.
\end{lemma}

\begin{proof}
For every vertex $v_i\in V(G)$, let $v_{i0},...,v_{in-1}$ denote the corresponding vertices in $KG.$ Let $K'$ be an edge-labeling of $KG$ assigning the permutation $K(e)$ to all edges corresponding to $e.$ The vertex-assignment $k':V(KG)\mapsto[n],$ in which $k'(v_{ij})=j,$ is consistent. 
\end{proof}


\begin{figure}[h]
\begin{center}
\begin{picture}(300,150)

\put(30,50){\line(0,1){50}}
\put(10,72){(12)}
\put(30,100){\line(1,0){50}}
\put(45,103){(12)}
\put(80,50){\line(-1,0){50}}
\put(45,40){(02)}
\put(80,100){\line(0,-1){50}}
\put(82,72){(01)}
\put(30,50){\circle*{3}}
\put(25,45){$v_0$}
\put(30,100){\circle*{3}}
\put(25,102){$v_3$}
\put(80,50){\circle*{3}}
\put(82,45){$v_1$}
\put(80,100){\circle*{3}}
\put(82,102){$v_2$}

\qbezier(190,100)(190,100)(190,50)
\qbezier(190,100)(190,100)(240,100)
\qbezier(240,50)(240,50)(150,10)
\qbezier(240,100)(240,100)(260,30)

\put(190,50){\circle*{3}}
\put(192,50){$v_{00}$}
\put(190,100){\circle*{3}}
\put(190,102){$v_{30}$}
\put(240,50){\circle*{3}}
\put(240,45){$v_{10}$}
\put(240,100){\circle*{3}}
\put(240,102){$v_{20}$}

\qbezier(170,30)(170,30)(150,140)
\qbezier(170,120)(170,120)(280,140)
\qbezier(260,120)(260,120)(240,50)
\qbezier(260,30)(260,30)(170,30)

\put(170,30){\circle*{3}}
\put(172,32){$v_{01}$}
\put(170,120){\circle*{3}}
\put(172,115){$v_{31}$}
\put(260,30){\circle*{3}}
\put(262,30){$v_{11}$}
\put(260,120){\circle*{3}}
\put(260,122){$v_{21}$}

\qbezier(150,10)(150,10)(170,120)
\qbezier(150,140)(150,140)(260,120)
\qbezier(280,140)(280,140)(280,10)
\qbezier(280,10)(280,10)(190,50)

\put(150,10){\circle*{3}}
\put(152,8){$v_{02}$}
\put(150,140){\circle*{3}}
\put(150,142){$v_{32}$}
\put(280,10){\circle*{3}}
\put(282,10){$v_{12}$}
\put(280,140){\circle*{3}}
\put(282,142){$v_{22}$}

\put(35,20){$G$}
\put(210,10){$KG$}

\end{picture}
\end{center}

\caption{The graphs $G$ and $KG$. The vertex-assignment $k:V(KG)\mapsto [3],$ 
where $k(v_{ij})=j$ is consistent for $(KG,K')$. There is no consistent
vertex-assignment for $(G,K)$.}
\label{fig:example}
\end{figure}
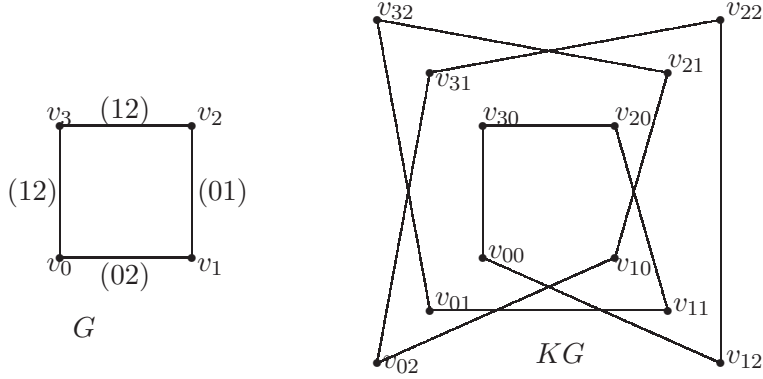

\begin{theorem}
The graph $G$ has a consistent vertex-labeling for a given edge-labeling $K$
if and only if $KG$ has a component isomorphic to $G$. If $G$ is connected
then the number of consistent vertex-labelings is equal to the number of
such components.
\end{theorem}

\begin{proof}
Let $V_i=\{v_{i0},...,v_{in-1}\}$ denote the set of vertices in $KG$ corresponding to $v_i\in V(G).$ For any $V_i,V_j$ and vertex $u\in V_i,$ $\left|N_{V_j}(u)\right|=1$ if $v_j\in N_G(v_i)$ and $\left|N_{V_j}(u)\right|=0$ if $v_j\notin N_G(v_i).$ Therefore, for any component $H$ of $KG,$ if $\left|V_i\cap V(H)\right|=c,$ then $\left|V_j\cap V(H)\right|=c,$ for all $j$ such that $v_j\in N_G(v_i).$ If $G$ is a connected graph, then for any component $H$ of $KG$ $\left|V_i\cap V(H)\right|=\left|V_j\cap V(H)\right|$ for any $i$ and $j.$
Thus, a component isomorphic to $G$ must contain exactly one vertex corresponding to each $v\in V(G)$ and one edge corresponding to each $e\in E(G).$ By Lemma~1, this component has a consistent vertex-assignment. 
Now, let $k: V(G)\mapsto[n]$ be a consistent vertex-assignment of $G.$ For every $v\in V(G),$ $k(v)$ determines $k(u)$ for all $u\in N_G(v).$
For any $v_i\in V(G),$ there exists a vertex $v_{ik(v_i)}\in V_i.$ For this vertex $k'(v_{ik(v_i)})=k(v_i)$ and for all $u_{jh}\in N_{KG}(v_{ik(v_i)}),$ $k'(u_{jh})=k(u_j).$
Thus, each consistent assignment of $G$ defines a component of $KG$ isomorphic to $G.$
\end{proof}

\begin{cor}
If $G$ is a connected graph then any component of $KG$ with the same number
of vertices as $G$ is isomorphic to $G$.
\end{cor}

\begin{proof}
Let $G$ be a connected graph and let $V_i=\{v_{i0},...,v_{in-1}\}$ denote the set of vertices in $KG$ corresponding to $v_i\in V(G).$ For any $V_i,V_j$ and vertex $u\in V_i,$ $\left|N_{V_j}(u)\right|=1$ if $v_j\in N_G(v_i)$ and $\left|N_{V_j}(u)\right|=0$ if $v_j\notin N_G(v_i).$ Therefore, for any pair $V_i,V_j$ and any component $H$ of $KG,$ $\left|H\cap V_i\right|=\left|H\cap V_j\right|.$ Thus, any component with $\left|V(G)\right|$ vertices contains exactly one vertex from each $V_i$ and it is isomorphic to $G.$
\end{proof}

Let us consider the case where $K:E\mapsto S_{3}$. 

\begin{theorem}
For any edge-labeling $K:E(G)\mapsto S_{3},$ any connected graph $G$ has
either $3,1,$ or $0$ consistent vertex-labelings.
\end{theorem}

\begin{proof}
For any component $H$ of $KG,$ $\left|V_i\cap V(H)\right|=\left|V_j\cap V(H)\right|,$ for all $V_i,V_j\subset V(KG).$ It follows that $KG$ can have either $3,1,$ or $0$ components isomorphic to $G.$
\end{proof}

For a given $K:E\mapsto S_{n},$ we say that a graph (or its subgraph) is 
\emph{good} if it has $n$ consistent vertex-labelings. If it has no
consistent vertex-labeling, we say that it is \emph{bad}. Otherwise, the
graph is \emph{ugly}.

\section{Operations on labeled graphs}

Obviously, for any graph $G$ and edge-labeling $K:E(G)\mapsto S_{n}$ if $%
\beta _{C}^{\prime }(G,K)>0$ then $\beta _{C}(G,K)=0$. The assignment and
contradiction numbers of various subgraphs of a given graph may also be
useful in calculating $\beta _{C}(G,K)$ and $\beta _{C}^{\prime }(G,K).$

\subsection{Subgraphs of labeled graphs}

For a labeled graph $(H,K_H)$ we consider a subgraph $G$ of $H$ with an edge-labeling
$K_G:E(G)\mapsto S_n$, where $K_G(e)=K_H(e)$ for all $e\in E(G).$ We compare the contradiction 
and assignment numbers of $(G,K_G)$ to those of $(H,K_H)$. This will help us study the 
effect of various graph operations on these parameters.

\begin{obs}
Let $H$ be a graph and $K_{H}$ an edge-labeling of $H$ with a set of
permutations. Let $G$ be a subgraph of $H$ and let the edge-labeling $K_{G}$
of $G$ be defined as $K_{G}(e)=K_{H}(e)$ for all $e\in E(G)$. Then $\beta
_{C}(G,K_{G})\leq \beta _{C}(H,K_{H}).$
\end{obs}

\begin{proof}
Let $k:V(H)\mapsto [n]$ be any vertex-assignment of $H$ and $k_G:V(G)\mapsto [n]$ a vertex assignment of $G$ such that $k_G(v)=k(v)$ for all $v\in V(G).$ It is clear that $k$ must have at least as many contradictions as $k_G.$ Thus, $\beta_C(G,K_G)\leq\beta_C(H,K_H).$
\end{proof}

\begin{obs}
Let $H$ be a connected graph and $K_{H}$ an edge-labeling of $H$ with a set
of permutations. Let $G$ be a subgraph of $H$ and let the edge-labeling $%
K_{G}$ of $G$ be defined as $K_{G}(e)=K_{H}(e)$ for all $e\in E(G)$. Then $%
\beta _{C}^{\prime }(G,K_{G})\geq \beta _{C}^{\prime }(H,K_{H})$.
\end{obs}

\begin{proof}
Let $k:V(H)\mapsto [n]$ be any consistent vertex-assignment of $H.$ The assignment $k_G:V(G)\mapsto [n]$ of $G$ such that $k_G(v)=k(v)$ for all $v\in V(G)$ must also be consistent. In a consistent vertex-assignment of a connected graph the value assigned to one vertex determines the values in all other vertices. Thus, for any consistent $k_G$ there is at most one $k$ such that $k_G(v)=k(v).$
\end{proof}

\begin{theorem}
For any graph $G,$ edge-labeling $K:E(G)\mapsto S_{n}$ and edge $e\in E(G)$,
we have%
\begin{equation*}
\beta _{C}(G,K)-1\leq \beta _{C}(G-e,K)\leq \beta _{C}(G,K).
\end{equation*}
\end{theorem}

\begin{proof}
If the edge $e$ is a contradiction in some optimal vertex-assignment $k:V(G)\mapsto[n],$ then in $G-e$ the assignment $k$ is still optimal but has $\beta_C(G,K)-1$ contradictions.
Otherwise every optimal vertex-assignment of $G$ is also an optimal assignment of $G-e$ and $\beta_C(G,K)=\beta_C(G-e,K).$
\end{proof}

\begin{obs}
For any graph $G,$ edge-labeling $K:E(G)\mapsto S_{n}$ and edge $e\in E(G)$,
we have%
\begin{equation*}
\beta _{C}^{\prime }(G-e,K)\geq \beta _{C}^{\prime }(G,K).
\end{equation*}
\end{obs}

\begin{proof}
It follows from the above theorem that any consistent vertex-assignment of $G$ remains consistent for $G-e.$
\end{proof}

\subsubsection{Identifying two vertices of a labeled graph}

Let $(G,K_G)$ be a labeled graph and let $v_1$ and $v_2$ be two vertices of $G$. By $H=G_{v_1=v_2}$ we denote the graph obtained by identifying $v_1$ and $v_2$.
The vertices $v_1$ and $v_2$ are replaced by a new vertex $v$, that is, $\ V(H)=(V(G)-\{v_{1},v_{2}\})\cup \{v\}$. 
The edge set of $H$ is defined as follows.

\begin{enumerate}

\item $N_H(v)=(N_G(v_1)\cup N_G(v_2)) - \{v_1,v_2\};$

\item $N_H(u)=N_G(u)\cup\{v\} - \{u,v\}$ for $u\in N_H(v);$

\item $N_H(u)=N_G(u)$ for $u\notin N_H(v).$

\end{enumerate}

Let $K_{H}\colon E(H)\mapsto S_{n}$ be an inherited edge-labeling, i.e., an edge-labeling of $H$ in which: 

\begin{enumerate}
\item $K_{H}(u_{1}u_{2})=K_G(u_{1}u_{2})$ for $u_1,u_2\in V(H)-\{v\}$; 

\item If $u\in N_H(v),$ then $K_H(uv)=K_G(uv_i),$ where $uv_i\in E(G).$

\end{enumerate}

\begin{theorem}
\label{Tidcontr}
For any labeled graph $(G,K_G)$ and pair of vertices $v_{1},v_{2}\in V(G),$ the contradiction number of $(H,K_H)$, where $H=G_{v_1=v_2}$ and $K_H$ is the inherited edge-labeling, satisfies the following inequalities:
\begin{equation*}
\beta _{C}(G,K_{G})-1\leq \beta _{C}(H,K_{H})\leq \beta _{C}(G,K_{G})+\min
\{\deg _{G}v_{1},\deg _{G}v_{2}\}.
\end{equation*}
\end{theorem}
\begin{proof} 
Assume without loss of generality that $\deg_G(v_1)\geq\deg_G(v_2)$. For a given vertex-assignment $k\colon V(G)\mapsto [n]$, we define a vertex-assignment $k'\colon V(H)\mapsto [n]$, where $k'(v)=k(v_1)$ and $k'(u)=k(u)$ for $u\in V(H)-\{v\}.$
 
If there exists an optimal vertex-assignment $k$ of $(G,K)$ such that $k(v_1)=k(v_2),$ then $k'$ is an optimal assignment in which $k'(v)=k(v_1)=k(v_2).$ Consequently, if $v_1v_2$ is a contradiction in $(G,K_G),$ then $\beta_C(H,K_H)=\beta_C(G,K_G)-1$ and, if $v_1v_2$ is not a contradiction, then $\beta_C(H,K_H)=\beta_C(G,K_G).$

If $k$ is an optimal vertex-assignment of $(G,K_G)$ in which $k(v_1)\neq k(v_2),$ then $k'$ is a vertex-assignment of $(H,K_H)$ with no more than $\beta_C(G,K_G)-1$ and at most $\beta_C(G,K_G)+\deg_G(v_2)$ contradictions. It follows that $\beta _{C}(G,K_{G})-1\leq \beta _{C}(H,K_{H})\leq \beta _{C}(G,K_{G})+\min\{\deg _{G}v_{1},$$\deg _{G}v_{2}\}$.


\end{proof}

The arguments used in the proof of Theorem \ref{Tidcontr} also imply that if $\beta_C'(G,K_G)>0$, then the assignment number of $(H,K_H)$ is equal to the number of consistent vertex-assignments $k$ of $(G,K_G)$ such that $k(v_{1})=k(v_{2}).$


We now consider labeled graphs $(G,K_G)$ which are not connected. Let $G_1$ and $G_2$ denote two subgraphs of $G$ such that $E(V(G_1),V(G_2))=\emptyset$ and let $K_1\colon E(G_1)\mapsto S_n$ and $K_2\colon E(G_2)\mapsto S_n$ be edge labelings of $G_1$ and $G_2,$ respectively, in which $K_i(e)=K_G(e)$ for all edges $e$ of $G_i$.

\begin{theorem}
\label{Tidass}
Let $(G,K_G)$ be a disconnected labeled graph and let $v_1\in V(G_1)$ and $v_2\in V(G_2)$ be two vertices of $G$ not connected by a path. If $H=G_{v_1=v_2}$ and $K_H$ is the inherited edge-labeling of $H$, then the following inequalities hold:

\begin{equation*}
\beta _{C}^{\prime }(G_{1},K_{1})+\beta _{C}^{\prime }(G_{2},K_{2})-n\leq
\beta _{C}^{\prime }(H,K_{H})\leq \min \{\beta _{C}^{\prime}(G_{1},K_{1}),\beta _{C}^{\prime }(G_{2},K_{2})\}.
\end{equation*}

\end{theorem}
\begin{proof}
Observation 3 implies that $\beta'_C(H,K_H)\leq\beta'_C(G_1,K_1)$ and $\beta'_C(H,K_H)\leq\beta'_C(G_2,K_2)$. It follows that $\beta'_C(H,K_H)\leq\min\{\beta'_C(G_1,K_1),\beta'_C(G_2,K_2)\}.$

Let us note that for a given vertex $u$ and $t\in [n]$ a labeled graph has at most one consistent vertex-assignment $k$ such that $k(u)=t.$ 

Now let $A_i$ for $i\in \{1,2\}$ denote the set of values $t\in [n]$ for which consistent vertex-assignments $k_i^t\colon V(G_i)\mapsto [n]$ such that $k_i^t(v)=t.$ The assignment $k^t\colon V(H)\mapsto [n]$ such that $k^t(v)=t$, for $t\in [n]$, exists if and only if $t\in A_1\cap A_2$. It follows that, $\beta'_C(H,K_H)=\left|A_1\cap A_2\right|$.

It is clear that $\left|A_1\cup A_2\right|\leq n$ and that $\left|A_1\cap A_2\right|=\left|A_1\right|+\left|A_2\right|-\left|A_1\cup A_2\right|.$ Thus, $\beta'_C(H,K_H)=\beta'_C(G_1,K_1)+\beta'_C(G_2,K_2)-\left|A_1\cup A_2\right|$, and therefore $$\beta'_C(H,K_H)\geq\beta'_C(G_1,K_1)+\beta'_C(G_2,K_2)-n$$
\end{proof}

Since $\beta_C'(G,K)>0$ implies $\beta_C(G,K)=0,$ we have the following corollary.

\begin{cor}
Let $G$ be a disconnected graph with components $G_1$ and $G_2$ and, let $K_1$ and $K_2$ be edge-labelings of $G_1$ and $G_2$, respectively, $H=G_{v_1=v_2}$ for $v_1\in V(G_1), v_2\in V(G_2)$ and let $K_H$ be the inherited edge-labeling of $G$. 
If $\beta _{C}^{\prime }(G_{1},K_{1})+\beta _{C}^{\prime }(G_{2},K_{2})>n$, then $\beta _{C}(H,K_{H})=0$. 
\end{cor}

\begin{proof}
If $\beta'_C(G_1,K_1)+\beta'_C(G_2,K_2)>n,$ then $0<\beta'_C(G_1,K_1)+\beta'_C(G_2,K_2)-n\leq\beta'_C(H,K_H)$ and therefore $\beta_C(H,K_H)=0.$
\end{proof}

\section{Labeled graph equivalence}
\label{equiv}

A signed graph $\Sigma $ consists of a graph $G$ and a labeling $\hat{K}%
:E(G)\mapsto \{+,-\}.$ We say that a cycle of $\Sigma $ is \emph{balanced} if
it contains an even number of negative ($"-"$) edges. A signed graph is 
\emph{in balance} if all of its cycles are balanced. It is easy to see that
this corresponds directly to an edge-labeling $K:E(G)\mapsto S_{2}$, where $%
K(e)=id$ if $\hat{K}(e)=+$ and $K(e)=(01)$ otherwise. A signed graph $\Sigma
=(G,\hat{K})$ is in balance if and only if $\beta _{C}(G,K)=0$. Two signed
graphs $\Sigma _{1}$ and $\Sigma _{2}$ are said to be \emph{equivalent} if
they have the same underlying graph $G$ and the same set of balanced cycles.
The \emph{switching operation} is defined by changing the signs $\hat{K}(e)$
of all edges incident to a certain vertex $v\in V(G).$ Two signed graphs are
equivalent if and only if one can be obtained from the other by some
isomorphism of the underlying graphs and a sequence of switches. We can
define a similar equivalence for labeled graphs. Since different labelings
of the same graph $G$, such that $\beta _{C}$ and $\beta _{C}^{\prime }$ of
each cycle are the same, do not necessarily need to have the same $\beta _{C}
$ or $\beta _{C}^{\prime }$ for the whole graph, an equivalence based on
that would be meaningless. Instead, we consider two labeled graphs $%
(G_{1},K_{1})$ and $(G_{2},K_{2})$ to be equivalent if one can be obtained
from the other one by 
an isomorphism between $G_{1}$ and $G_{2},$ 
replacing a directed edge $uv$ labeled with $\pi $ with an edge $vu$ labeled with 
$\pi^{-1}$, 
and switching operations $s(v,\sigma )$ defined as follows. 
For any graph $G$ and edge-labeling $K:E(G)\mapsto S_{n}$, let $v\in V(G)$ be any
vertex of $G$ and let $\sigma $ be any permutation of $[n]$. 
For every vertex $u\in N_{G}(v)$: 

\begin{enumerate}
\item if $uv\in E,$ we replace $K(uv)=\pi$ with $K'(uv)=\sigma\pi$,

\item if $vu\in E,$ we replace $K(vu)=\pi$ with $K'(vu)=\pi\sigma^{-1}$.
\end{enumerate}

Here, $N_{G}(v)$ is the open neighborhood of $v$ in $G$. 


\begin{theorem}
For any given $G_1, G_2, K_1:E(G_1)\mapsto S_n, K_2:E(G_2)\mapsto S_n$ the
labeled graphs $(G_1,K_1)$ and $(G_2,K_2)$ are equivalent if and only if
there exists an isomorphism $F:V(K_1G_1) \mapsto V(K_2G_2)$ such that 
for every $v_i\in V(G_1)$ there exists a $u_i\in V(G_2)$ such that 
$F(\{v_{i0},...,v_{in-1}\})=\{u_{i0},...,u_{in-1}\}$
\end{theorem}

{\begin{proof}
First note that if $\pi(x)=y,$ then $\pi^{-1}(y)=x.$ It folows that changing
a directed edge $uv\in E(G_1)$ labeled with a permutation $K_1(uv)=\pi$ to $%
vu$ labeled with $\pi^{-1}$ results in $(G_2,K_2)$ identical to $(G_1,K_1)$.

Suppose $(G_2,K_2)$ can be obtained from $(G_1,K_1)$ by some isomorphism 
$f:V(G_1)\mapsto V(G_2)$ and a series of switches $s_1,...,s_t$ where 
$s_k=s(v_k,\sigma_k).$ We can define functions $\hat{f},%
\hat{s_1},...,\hat{s_t}:V(K_1G_1)\mapsto V(K_2G_2)$ as follows: $\hat{f}%
(v_{ij})= v_{f(i)j}$ and

\begin{equation*}
\hat{s_k}(v_{ij}) =\begin{cases}
v_{ij}, & i\neq k;
\\
v_{k\sigma_k(j)}, & i=k.
\end{cases}
\end{equation*}

It is easy to see that the function $F=\hat{f}\circ\hat{s_1}\circ ...\circ%
\hat{s_t}$ is an isomorphism from $K_1G_1$ to $K_2G_2$ and that for every 
$v_i\in V(G_1)$ there exists a $u_i\in V(G_2)$ such that 
$F(\{v_{i0},...,v_{in-1}\})=\{u_{i0},...,u_{in-1}\}$

Now, let $F:V(K_1G_1)\mapsto V(K_2G_2)$ be an isomorphism such that $F(V_i)=U_j,$ where 
$V_i=\{v_{i0},...,v_{in-1}\}\subset V(K_1G_1)$ and $U_i=\{u_{i0},...,u_{in-1}\}\subset V(K_2G_2)$.
This isomorphism can be presented as $F=\hat{f}%
\circ\hat{s_1}\circ ...\circ\hat{s_t}$ such that $\hat{f}(v_{ij})= v_{i'j}$
and

\begin{equation*}
\hat{s_k}(v_{ij}) =\begin{cases}
v_{ij}, & i\neq k;
\\v_{kj'}, & i=k;
\end{cases}
\end{equation*}

where $i'=f(i)$ for some isomorphism $f:G_1\mapsto G_2$ and $j'=\sigma_k(j)
$ for some permutation $\sigma_k\in S_n.$
\end{proof}
}

It follows that equivalent labeled graphs must have the same contradiction
and assignment numbers. All equivalent labelings of a given graph $G$ must
also have an optimal vertex-labeling with the same set of contradictions.

\section{Labeling with $S_2$}

Labeling with $S_{2}$ is equivalent to the problem of balance in a signed
graph. In \cite{Har} Harary proved the following result.

\begin{theorem}[Harary \cite{Har}]
\label{THar}
A signed graph is balanced if and only if its set of vertices can be partitioned 
into two disjoint subsets in such a way that each positive edge joins two vertices 
in the same subset while each negative edge joins two vertices from different subsets.
\end{theorem}

It follows that a signed graph is \emph{in balance} if and only if the corresponding 
labeled graph has a consistent vertex-assignment.

Edge-labelings with $S_{2}$ correspond to a class of unique games 
known as XOR games, in which the value of each
variable can be either $0$ or $1$ and the constraints are given in the form
of the XOR of those values, \emph{i.e.} the constraint is satisfied if and
only if $x+y \equiv i \Mod{2}$, where $i$ is either $0$ (\emph{i.e.} $\pi =id$%
) or $1$ ($\pi =(01)$). It is easy to see that for the edge-labeling $%
K:E\mapsto S_{n}$ of any graph $G$ such that $K(e)=id$ for every $e\in E,$ $%
\beta _{C}(G,K)=0,$ $\beta _{C}^{\prime }(G,K)=n$ and the only consistent
vertex-assignments are those which assign the same value to every vertex.
Let us consider the edge-labeling $K:E\mapsto S_{2}$ in which $K(e)=(01)$
for all $e\in E$. 

\begin{prop}
The edge-labeling $K:E\mapsto \{(01)\}\subset S_{2}$ (\emph{i.e.} a labeling
using only the permutation $\pi =(01)$) of a given graph $G\neq \overline{%
K_{n}}$ is vertex-proper if and only if $G$ is bipartite.
\end{prop}

\begin{proof}
Let $k:V\mapsto[2]$ be a consistent vertex assignment. Let $A=\{v\in V:c(v)=0\}$ and $B=\{v\in V:k(v)=1\}.$ The two sets are independent and $A\cup B=V.$ Thus $G$ is a bipartite graph.
Let $G$ be a bipartite graph. The vertex set of $G$ can be paritioned into two independent subsets $A$ and $B.$ The vertex assignments $k$ and $k'$ such that $k(v)=0$ if $v\in A$ and $k(v)=1$ if $v\in B$ and $k'(v)=0$ if $v\in B$ and $k'(v)=1$ if $v\in A$ are the only consistent vertex-assignments.
\end{proof}

The next fact follows directly from Theorem 8.

\begin{prop}
Let $K:E\mapsto \{(ab)\}\subset S_{n}$ for $n\geq 3$ be an edge labeling of
a non-bipartite graph $G\neq \overline{K_{n}}$ in which $K(e)=(ab)$ for all $%
e\in E.$ The only consistent vertex-labelings are $k:V\mapsto \{c\}\subset
\lbrack n],$ where $c$ is a fixed point of the permutation $(ab)$. If $G$ is
bipartite, a consistent vertex-labeling using $a$ and $b$ exists.
\end{prop}

What can we say about $\beta _{C}(G,(01))$, for a generic graph $G$? Let us
assume that $G$ is non-bipartite. Then, $\beta _{C}(G,(01))$ is the minimum
number of edges that we need to delete from $G$ to obtain a bipartite graph.
Computing $\beta _{C}(G,(01))$ is at least as hard as asking whether a
bipartite graph can be obtained by deleting at most $k$ edges from $G$. This
problem is called \textsc{Edge Bipartization} (or, equivalently,
(unweighted) \textsc{Minimum Uncut}). The problem is a way to measure how
close $G$ is to being bipartite. It is known to be MaxSNP-hard and can be
approximated to a factor of O$(\sqrt{\log n})$ in polynomial time, where $n$
is the total number of vertices (see \cite{ag}). Because of this point, we
can simplify the notation: the \emph{edge bipartization number} of a graph $G
$, $\beta _{C}^{2}(G)$, is the minimum number of edges that we need to
delete in order to obtain a bipartite graph. By the above definitions, $%
\beta _{C}(G,(01))=\beta _{C}^{2}(G)$. 

\section{Edge-labelings with Latin squares}

A \emph{Latin square} $L$ of order $n$ is an $n\times n$ array $%
\{l_{ij}\}_{n\times n}$ such that $l_{ij}\neq l_{is}$ and $l_{ij}\neq l_{tj}$
for $s\neq j$ and $t\neq i.$ A an $n\times n$ Latin square defines a set of $%
n$ permutations $\pi _{i}:[n]\mapsto \lbrack n].$ For each row in the array $%
\pi _{i}$ is defined by the set of pairs$\{(0,l_{i0}),...,(n-1,l_{in-1})\}$.
Let us consider a particular set $L_{n}=\{\pi _{0},...,\pi _{n-1}\}$ of
permutations where, for every $i\in \lbrack n]$ we have $\pi _{i}(x) \equiv i-x
\Mod{n}$. Note that $L_{n}$ corresponds to the Latin square in which $%
l_{ij} \equiv i-j \Mod{n}$, for all $i,j\in \lbrack n]$. In any set of permutations
defined by a Latin square, we have $\pi _{i}(x)\neq \pi _{j}(x)$ if $i\neq j$%
. Since every element $i\in \lbrack n]$ occurs exactly once in each column
of a given Latin square $L$, clearly every $i\in \lbrack n]$ is a fixed
point of exactly one permutation in the set defined by $L.$ If the set of
permutations contains the identity, no other permutation may have a fixed
point. Furthermore, if $n$ is odd, then each permutation in $L_{n}$ has
exactly one fixed point. 

\begin{obs}
For any edge-labeling $K:E\mapsto L_{n},$ a cycle of even length can have
either $0$ or $n$ consistent vertex-labelings.
\end{obs}

\begin{proof}
Let $L'_n=\{\sigma_0,...,\sigma_{n-1}\}$ be a set of permutations such that $\sigma_i(x) \equiv i+x \Mod{n}.$ It is clear that for any two permutations $\pi_i,\pi_j\in L_n,$ $\pi_j\pi_i\in L'_n,$ since $\pi_j\pi_i(x)=j-(i-x) \equiv (j-i)+x \Mod{n}.$ 
Now, for any $\sigma_i,\sigma_j\in L'_n,$ we have $\sigma_j\sigma_i\in L'_n,$ since $\sigma_j\sigma_i(x)=j+(i+x) \equiv (i+j)+x \Mod{n}.$ It follows that $\sigma_k...\sigma_1\in L'_n$ for any finite $k.$
Now let $C_{2k}$ be any even cycle and let $K:E(C)\mapsto L_n$ be an edge-labeling of $C_{2k}.$ Let $\pi_C=\pi_{2k}...\pi_1,$ where $\pi_i$ is the permutation assigned to the $i$-th edge of $C_{2k}.$ Obviously, $\pi_C=\sigma_k...\sigma_1,$ where $\sigma_i=\pi_{2i}\pi_{2i-1}.$ Since $\sigma_i\in L'_n$ for $i=1,2,...,k,$ this implies that $\pi_C\in L'_n.$ The only permutation in $L'_n$ which has a fixed point is $\sigma_0=id.$ Thus, either $\pi_C=id$ and $C_{2k}$ has $n$ consistent vertex-assignments or $\pi_C$ has no fixed point and $C_{2k}$ has no consistent assignments.
\end{proof}

\begin{obs}
For any edge-labeling $K:E\mapsto L_{n},$ for $n\geq 3$ a cycle of odd
length has at most $2$ consistent vertex-labelings.
\end{obs}

\begin{proof}
Let $L'_n=\{\sigma_0,...,\sigma_{n-1}\}$ be a set of permutations such that $\sigma_i(x) \equiv i+x \Mod{n}.$ 
Now, let $C_{2k+1}$ be a cycle of odd length. and let $K:E(C)\mapsto L_n$ be an edge-labeling of $C_{2k+1}.$ Let $\pi_C=\pi_{2k+1}...\pi_1,$ where $\pi_i$ is the permutation assigned to the $i$-th edge of $C_{2k+1}.$ Then $\pi_C=\pi_{2k+1}\sigma,$ where $\sigma=\pi_{2k}...\pi_1.$ By the above observation, $\sigma\in L_n'.$ For all $x\in [n]$ we have $\pi_{2k+i}(x) \equiv i-x \Mod{n}$ and $\sigma(x) \equiv j+x \Mod{n}$ for some $i,j\in [n].$ Thus $\pi(\sigma(x))=i-(j+x) \equiv (i-j)+x \Mod{n}$ and therefore $\pi_C\in L_n.$ It follows that $C_{2k+1}$ has exactly one vertex-assignment if $n$ is odd and either zero or two if $n$ is even.
\end{proof}

An interesting set of permutations is $L_{n}^{\prime },$ in which for every $%
i\in \lbrack n]$ we have $\sigma (x)_{i}=i+x.$ Here $\sigma ^{-1}\neq \sigma 
$ except for $\sigma _{0}=id,$ so we can only consider those labelings on a
directed graph. A contradiction in a given vertex-labeling $k$ is now a
directed edge $uv$ such that $\sigma (k(u))\neq k(v),$ where $\sigma $ is
the permutation assigned to the edge $uv.$

\begin{obs}
Let $G$ be a directed graph and let $K:E(G)\mapsto L_{n}^{\prime }$ be an
edge-labeling of $G.$ The graph $G$ can have either $0$ or $n$ consistent
vertex-labelings.
\end{obs}

\begin{proof}
A graph $G$ can only be ugly if it contains ugly cycles. Thus, let us consider a cycle $C_t=\{v_1,...,v_t,v_{t+1}=v_1\}\in G$ and let $\sigma_i$ denote the permutation assigned to the $i$-th edge of $C_n.$ Let $\pi_{C_t}=\pi'_t...\pi'_1,$ where $\pi_i'=\pi_i$ if $e_i=v_iv_{i+1}$ and $\pi_i'=\pi_i^{-1}$ if $e_i=v_{i+1}v_i.$ The assignment number of the $C_t$ is equal to the number of fixed points of $\pi_{C_t}.$ 
For any $\sigma_1,\sigma_2\in L_n'$ we have $\sigma_2\sigma_1,\sigma_1^{-1}\in L_n',$ so $\pi_{C_t}\in L_n'.$ The only permutation is $L_n'$ which has a fixed point is the identity. Thus, either $\pi_{C_t}=id$ and $C_t$ is good or $\pi_{C_t}$ has no fixed point and $C_t$ is bad.

\end{proof}

\subsection{Bipartite graphs}

A bipartite graph represents a two-player game. The referee asks each of the
players, who are not allowed to communicate, to assign a value from $[n]$ to
a variable. The players win if their answers satisfy the constraint $%
\pi(a)=b.$ The value of such a game is the maximum probability of winning.

\begin{theorem}
For any edge-labeling $K:E\mapsto L_n$ a bipartite graph $G$ is bad if and
only if it contains a bad chordless cycle.
\end{theorem}

\begin{proof}
Let $G$ be a bipartite graph. Obviously, if $G$ has no cycles, it has to be good. By Observation $6$, the graph $G$ contains no ugly cycles. Thus, $G$ is bad if and only if it contains a bad cycle. We will now prove that if $G$ contains a bad cycle, it must contain a bad chordless cycle.
Now, let $C_n=v_1...v_n$ for some $n\geq 6$ be a bad cycle in $G.$ Let $e_i=v_iv_{i+1}$ for $i\in\{1,...,2k-1\}$, $e_{2k}=v_{2k}v_1$ and $K(e_i)=\pi_i.$ Assume that an edge $e_{n+1}=v_1v_k$ exists, for some $k\in\{4,6,...,n-2\}.$ Let $\pi_{n+1}=K(e_{n+1}).$ Suppose both $v_1v_2...v_k$ and $v_kv_{k+1}...v_1$ are good cycles. Then $\pi_{n+1}\pi_k\pi_{k-1}...\pi_1=id$ and $\pi_{n+1}\pi_n...\pi_{k+1}=id.$ It follows that $\pi_{n+1}=\pi_n...\pi_{k+1}$ and therefore $\pi_n...\pi_{1}=id,$ a contradiction.
\end{proof}

In the complete bipartite graph $K_{s,t}$ all chordless cycles have length $4
$ and hence we have the following:

\begin{cor}
For any edge-labeling $K:E\mapsto L_{n}$, a complete bipartite graph $K_{s,t}
$ is bad if and only if it contains a bad $4$-cycle.
\end{cor}

\begin{proof}
It follows from Observation 6 that bipartite graphs contain no ugly cycles. Thus $K_{s,t}$ is bad if and only if it contains at least one bad cycle. First, we will prove that if $K_{s,t}$ contains a bad $6$-cycle, it must also contain a bad $4$-cycle. 
Let $v_1v_2v_3v_4v_5v_6$ be a cycle of length $6,$ where $e_i=v_iv_i+1$ for $i\in\{1,...,5\}$, $e_6=v_6v_1$ and $K(e_i)=\pi_i.$ Since $K_{s,t}$ is a complete bipartite graph, an edge $e_7=v_1v_4$ exists. Let $K(e_7)=\pi_7.$
Suppose all $4$-cycles in $K_{s,t}$ are good. Then $\pi_7\pi_3\pi_2\pi_1=id$ and $\pi_7^{-1}\pi_6\pi_5\pi_4=id.$ It follows that $\pi_6\pi_5\pi_4=\pi_7.$ Therefore $\pi_6\pi_5\pi_4\pi_3\pi_2\pi_1=id$ and the $6$-cycle $v_1v_2v_3v_4v_5v_6$ is good.
Thus a bad $6$-cycle can only exist if there is a bad 4-cycle. 
Now, let $v_1...v_{2k}$ be a cycle, where $k>3.$ Let $e_i=v_iv_{i+1}$ for $i\in\{1,...,2k-1\}$, $e_{2k}=v_{2k}v_1$ and $K(e_i)=\pi_i.$ Assume that a bad $2k-2$-cycle can only exist if a bad $4$-cycle exists. Let $e_{2k+1}=v_1v_4$ and $K(e_{2k+1})=\pi_{2k+1}.$
If all $4$-cycles in the graph are good, then $\pi_{2k+1}\pi_3\pi_2\pi_1=id$ and $\pi_{2k+1}^{-1}\pi_{2k}\pi_{2k-1}...\pi_4=id.$ It follows that $\pi_{2k}\pi_{2k-1}...\pi_4=\pi_{2k+1}$ and therefore $\pi_{2k}...\pi_1=id.$ Thus, the $2k$-cycle must be good.
\end{proof}

\subsection{Other graphs}

Every non-bipartite graph contains at least one odd cycle. It follows, that
for any edge-labeling $K:E\mapsto L_n$ where $n\geq 3$ a non-bipartite graph
cannot be good.

\begin{obs}
Let $G$ be a non-bipartite graph and let $K:E\mapsto L_{n}$ for some $n\geq 3
$ be an edge-labeling of $G$. If $n$ is odd then $\beta _{C}^{\prime
}(G,K)\leq 1$. If $n$ is even then $\beta _{C}^{\prime }(G,K)\leq 2.$
\end{obs}

\begin{proof}
If $H$ is a subgraph of $G$ and $K'$ is an edge-labeling of $H$ such that $K'(e)=K(e),$ then $\beta_C'(G,K)\leq\beta_C'(H,K').$ By Observation 5 the assignment number of an odd cycle is at most $2$ for an even $n$ and $1$ for an odd $n.$ Thus, the assignment number of any graph containing an odd cycle must be at most $2$ for even $n$ and at most $1$ for odd $n.$
\end{proof}

It follows that the assignment number of a graph labeled with $L_{n}$ has
to be either $0,1,2$ or $n$.

\bigskip

\textbf{Acknowledgments.} The authors would like to thank David Roberson for
pointing out links with topological graph theory and Markos Karasamanis. This work has been
supported by the Royal Society and EPSRC. Part of this work has been carried out while the authors were at the Isaac Newton Institute for Mathematical Science, during the thematic programme \textquotedblleft Mathematical Challenges in Quantum Information\textquotedblright.

\end{document}